\documentclass{article}
\usepackage{amsfonts}
\usepackage{amssymb}
\usepackage{amsmath}
\usepackage{amsthm}
\newtheorem{definition}{Definition}[section]
\newtheorem{lemma}[definition]{Lemma}
\newtheorem{theorem}[definition]{Theorem}
\newtheorem{corollary}[definition]{Corollary}
\newtheorem{remark}[definition]{Remark}
\newcommand{\bvec}[1]{\mbox{\boldmath $#1$}}
\title{Bounds on Walsh coefficients by dyadic difference and a new Koksma-Hlawka type inequality for Quasi-Monte Carlo integration}
\author{Takehito Yoshiki
\thanks{Graduate School of Mathematical Sciences, The University of Tokyo, 3-8-1 Komaba, Meguro-ku,
Tokyo 153-8914 (yosiki@ms.u-tokyo.ac.jp). The work was supported by the Program for Leading Graduate Schools, MEXT, Japan.}}
\begin{document}
\pagestyle{myheadings}
\markboth{T.Yoshiki}{Bounds on Walsh coefficients}
\maketitle
\begin{abstract}
In this paper we give a new Koksma-Hlawka type inequality for Quasi-Monte Carlo (QMC) integration.
QMC integration of a function $f\colon[0,1)^s\rightarrow \mathbb{R}$ by a finite point set
 $\mathcal{P}\subset [0,1)^s$ is the approximation
 of the integral $I(f):=\int_{[0,1)^s}f(\mathbf{x})\,d\mathbf{x}$
by the average $I_{\mathcal{P}}(f):=\frac{1}{|\mathcal{P}|}\sum_{\mathbf{x} \in \mathcal{P}}f(\mathbf{x})$. 
We treat a certain class of point sets $\mathcal{P}$ called digital nets.
 A Koksma-Hlawka type inequality is an inequality bounding the integration error $\mathrm{Err}(f;\mathcal{P}):=I(f)-I_{\mathcal{P}}(f)$
by a bound of the form $|\mathrm{Err}(f;\mathcal{P})|\le C\cdot \|f\|\cdot D(\mathcal{P})$.
We can obtain a Koksma-Hlawka type inequality by estimating bounds on $|\hat{f}(\mathbf{k})|$,
where $\hat{f}(\mathbf{k})$ is a generalized Fourier coefficient with respect to the Walsh system. In this paper we prove bounds on Walsh coefficients $\hat{f}(\mathbf{k})$
 by introducing an operator called `dyadic difference' $\partial_{i,n}$.
By converting dyadic differences $\partial_{i,n}$ to derivatives $\frac{\partial }{\partial x_i}$, we get 
a new bound on $|\hat{f}(\mathbf{k})|$ for a function $f$ whose mixed partial derivatives up to order $\alpha$ in each variable are continuous.
This new bound is smaller than the known bound on $|\hat{f}(\mathbf{k})|$ under some condition.
The new Koksma-Hlawka inequality is derived 
using this new bound on the Walsh coefficients.
\end{abstract}
\section{Introduction and the main results}
Quasi-Monte Carlo(QMC) integration of a function $f\colon[0,1)^s\rightarrow \mathbb{R}$ by a finite point set $\mathcal{P}\subset [0,1)^s$ is the approximation of the integral $I(f):=\int_{[0,1)^s}f(\mathbf{x})\,d\mathbf{x}$
by the average $I_{\mathcal{P}}(f):=\frac{1}{|\mathcal{P}|}\sum_{\mathbf{x} \in \mathcal{P}}f(\mathbf{x})$
(see \cite{Dick_Pill}, \cite{Nied} and \cite{Sloan} for details). 
We want to find quadrature point sets $\mathcal{P}$ making the absolute value of the integration error $|\mathrm{Err}(f;\mathcal{P})|:=|I(f)-I_{\mathcal{P}}(f)|$ small for a set of functions $f$.
This problem is formulated as follows: We consider a function space $H$ with a norm $\| f\|_H$
and the worst case error $\sup_{\| f\|_H\le 1}|\mathrm{Err}(f;\mathcal{P})|$
 by a QMC rule using the point set $\mathcal{P}$ (for example, see \cite{Dick_Pill}, \cite{Hickernell} for details).
Then, it holds that, for any $f\in H$,
\begin{eqnarray}
\label{KHoriginal}
|\mathrm{Err}(f;\mathcal{P})|\le \| f\|_H\times \sup_{\| f\|_H\le 1}|\mathrm{Err}(f;\mathcal{P})|.
\end{eqnarray}
Thus in order to make $|\mathrm{Err}(f;\mathcal{P})|$ small,
 we have to obtain quadrature point sets $\mathcal{P}$ making the worst case error  $\sup_{\| f\|_H\le 1}|\mathrm{Err}(f;\mathcal{P})|$ small.

We often treat a point set $\mathcal{P}$ called `digital net' (for example, see \cite{Nied}).
A digital net $\mathcal{P}$ is defined as follows.
Let $n,m,b\ge 1$ be integers with $n\ge m$. Let $0\le h< b^{m}$ be an integer
and $C_1,\dots,C_s$ be $n\times m$ matrices over the finite group $\mathbb{Z}_b=\mathbb{Z}/b\mathbb{Z}$.
We write the $b$-adic expansion $h=\sum_{j=1}^mh_jb^{j-1}$ and take a vector $\mathbf{h}=(h_1,\dots,h_m)\in(\mathbb{Z}_b^m)^\top$,
where $h_{j}$ is considered to be an element in $\mathbb{Z}_b$.
For $1\le i\le s$, we define the vector $(y_{h,i,1},\dots,y_{h,i,n})=\mathbf{h}\cdot(C_i)^{\top}$ and 
a real number $x_i(h)=\sum_{1\le j\le n}y_{h,i,j}b^{-j}\in[0,1)$, 
where $y_{h,i,j}$ is considered to be an element of $\{0,\dots, b-1\}\subset\mathbb{Z}$.
Then we define a digital net $\mathcal{P}$ by $\{\mathbf{x}_0,\cdots,\mathbf{x}_{b^m-1}\}$ where $\mathbf{x}_h=(x_i(h))_{1\le i\le s}$.
We define the dual net $\mathcal{P}^\bot$ \cite{Dick_Pill05,NiedP}, which is essential to analyze the integration error:
\begin{eqnarray*}
\mathcal{P}^\bot
 :=\{ \mathbf{k}=(k_1,\dots, k_s)\in\mathbb{N}_0^s\mid
 C_1^{\top}\vec{k}_1+\cdots+C_s^{\top}\vec{k}_s=\mathbf{0}\in\mathbb{Z}_b^m\} ,
\end{eqnarray*}
where $\vec{k}_i=(\kappa _{i,1},\dots,\kappa _{i,n})^{\top}$ for
$k_i$ with $b$-adic expression $k_i=\sum_{j\ge 1}\kappa _{i,j}b^{j-1}$.
Here $\kappa _{i,j}$ is considered to be an element of $\mathbb{Z}_b$.
Throughout this paper, when we take a point set $\mathcal{P}$,
we assume that $\mathcal{P}$ is a digital net with $b=2$.

In the classical theory, many researchers studied the integration error of 
a function $f$ with bounded variation (or function with square integrable partial derivatives up to first order in each variable) (for example, see \cite{Dick_Pill}, \cite{Kuiper}).
An extension to smooth periodic functions was established
in \cite{Dick_periodic},
while a further extension to smooth (non-periodic) functions
was shown in \cite{Dick_nonperiodic}.
The QMC rules constructed in these papers, called higher order QMC rules,
achieve (up to powers of $\log N$) the optimal rate of convergence.
See also \cite{Dick_decay2}
for more background on higher order QMC rules.
The purpose of this paper is
to substantially improve the constants in the bounds
on integration error
in \cite{Dick_nonperiodic},
which is crucial in problems
in uncertainty quantification \cite{Dick_garkin, GS14}.
In particular, \cite{GS14} point out that
the large constants from \cite{Dick_nonperiodic} cause problems
in the CBC construction of interlaced polynomial lattice rules,
which is one of construction methods to obtain point sets whose worst case error achieves the optimal order.
To avoid this problem, they suggest to use much smaller constants
which are more realistic.
This paper provides the theoretical justification for doing so.

We explain the details.
Dick et al. \cite{Dick_garkin} introduced a smooth function space whose functions $f$ satisfy 
that their norms (\ref{dick_norm}) (see below) are finite.
If $f$ is a function whose mixed partial derivatives up to order $\alpha$ in each variable are continuous,
then $f$ is contained in this space.
This space has some parameters called weights
 $\{\gamma_v\}_{v\subset S}\subset\mathbb{R}_{>0}=\{ x\in \mathbb{R}: x>0 \}$
,where $S:=\{1,\dots, s\}$,
which model the importance of different coordinate projections, 
see \cite{SW98}.

To state their results,
we need modified dual spaces which correspond to the subsets $v\subset S$.
For $\mathbf{k}_v \in \mathbb{N}^{|v|}$,
 let $(\mathbf{k}_v; \mathbf{0})$ denote the vector whose $j$th component is $k_j$ if $j\in v$ and $0$ otherwise.
We define the dual space which corresponds to the subset $\phi\neq v\subset S$ by
$\mathcal{P}_v^\bot :=\{ \mathbf{k}_v\in\mathbb{N}^{|v|}\mid\mathbf{k}=(\mathbf{k}_v;\mathbf{0})\in\mathcal{P}^\bot\}$
(note that none of the components in $v$ is $0$).
Let $1\le r,r',q\le \infty$ with $1/r+1/r'=1$
and
\begin{eqnarray*}
\mu_{\alpha}((l_1,\dots,l_{|v|}))=\sum_{i=1}^{|v|}\sum_{j\le \alpha}(a_{i,j}+1)
\end{eqnarray*}
for $l_i$ with dyadic expansion $l_i=\sum_{j=1}^{N_i}2^{a_{i,j}}$, with $a_{i,1}>\cdots>a_{i,N_i}$.
They showed the following bound on the worst case error
(Dick et al. \cite{Dick_garkin} also showed the results for a digital net with $b\ge 2$):
\begin{eqnarray*}
\sup_{\| f\|_{s,\alpha,\gamma ,q,r}\le 1} 
|\mathrm{Err}(f;\mathcal{P})|\le
e_{s,\alpha,\gamma ,r'}(\mathcal{P}),
\end{eqnarray*}
with
\begin{eqnarray}
\label{wce_dick}
e_{s,\alpha,\gamma ,r'}(\mathcal{P})=
\left(\sum_{\phi \neq v\subset S}\left( C_{\alpha}^{|v|} \gamma_{v}\sum_{\mathbf{k}_{v}\in\mathcal{P}_{v}^\bot }2^{-\mu _{\alpha}(\mathbf{k}_{v})} \right)^{r'} \right) ^{1/r'}.
\end{eqnarray}
This implies the following inequality of the form (\ref{KHoriginal}):
\begin{eqnarray}
\label{KH_dick}
\left| \mathrm{Err}(f;\mathcal{P})\right|
\le 
\| f\|_{s,\alpha,\gamma ,q,r}\times
e_{s,\alpha,\gamma ,r'}(\mathcal{P}),
\end{eqnarray}
where
\begin{eqnarray}
\label{dick_norm}
 & & \|f\|_{s,\alpha , {\gamma},q,r}
 :=
 \Bigg( \sum_{ {u}\subseteq S} \Bigg( \gamma_ {u}^{-q}
 \sum_{ {v}\subseteq {u}} \sum_{\bvec{\tau}_{ {u}\setminus {v}}
 \in \{ 1,\dots ,\alpha -1 \}^{| {u}\setminus {v}|}} \\
 & &\qquad\qquad\quad
 \int_{[0,1]^{| {v}|}} \bigg|\int_{[0,1]^{s-| {v}|}} \!
(\partial ^{(\bvec{\alpha}_{ {v}}
,\bvec{\tau}_{ {u}\setminus {v}},\bvec{0})}
_{\bvec{y}} f)
(\bvec{y}) \,\, \mathrm{d} \bvec{y}
_{ S \setminus {v}}
 \bigg|^q \, \mathrm{d} \bvec{y}_ {v} \Bigg)^{r/q} \Bigg)^{1/r}, \nonumber
\end{eqnarray}
with the obvious modifications if $q$ or $r$ is infinite.
\footnote{The norm in \cite[Definition 3.3]{Dick_garkin} has been corrected 
in arXiv:1309.4624v3. The correct version is restated here in Eq. (\ref{dick_norm}).}
Here $(\alpha _v,\tau _{u\backslash v},\mathbf{0})$ denotes a sequence $(\nu_j)_j$ with $\nu_j=\alpha$ for $j\in v$, $\nu_j=\tau_j$ for $j\in u\backslash v$,
and $\nu_j=0$ for $j\not\in u$.
And we write $f^{(n_1,\dots,n_s )}=\partial ^{n_1+\cdots +n_s}f/\partial x_1^{n_1}\cdots\partial x_s^{n_s}$.

Based on these bounds on the integration error,
Dick constructed `interlaced digital nets' to obtain a point set
with small integration error (for example, see \cite{Dick_periodic},\cite{Dick_nonperiodic}).
He showed that the worst case error of
this type of point set achieves the order $O(N^{-\alpha}(\log N)^{s\alpha})$
in terms of the cardinality $N$ of a point set (see \cite{Dick_nonperiodic}).
This is known to be optimal up to log terms (see \cite{sharygin}). 
In \cite{BJJF,BJJGDF} there is also a component-by-component (CBC) algorithm
to obtain point sets which achieve the same order.

There is another algorithm to find good quadrature point sets for QMC
for integrands with large enough smoothness $\alpha$.
This was introduced by Matsumoto, Saito and Matoba \cite{MSM}. 
They define the Walsh Figure of Merit (WAFOM), which is defined by the discretization of the upper bound of the worst case error of the form (\ref{wce_dick}).
The advantage of WAFOM is that we can compute it on the computer in reasonable time.
This property enables us to find a point set with small integration error by computer search. 
In fact, there are some algorithms for finding good point sets for QMC (see \cite{Harase},\cite{HO}).

In this way, to find good quadrature point sets, we need an inequality of the form (\ref{KHoriginal}),
which bounds the integration error by the product of a norm of $f$ and a figure of merit of $\mathcal{P}$.
These types of inequalities are called Koksma-Hlawka inequalities
(for example, see \cite{Kuiper} for details).
In the following, we give a new Koksma-Hlawka type inequality to bound the integration error of
smooth functions better than the inequality (\ref{KH_dick}) under some condition.
\begin{theorem}
\label{yKH}
Let $\alpha\in\mathbb{N}\cup\{\infty\}$ such that $\alpha\ge 2$.
We assume that a function f satisfies that its mixed partial derivatives up to order $\alpha$ in each variable $x_i$ are continuous on $[0,1]^s$,
 and $1\le p, q, q'\le \infty$ such that $1/q +1/q' = 1$.
Then we have
\begin{eqnarray*}
\left|\mathrm{Err}(f;\mathcal{P})\right|
\le \| f\|_{\mathcal{B}_{\alpha},\gamma ,p,q'}
\times\mathcal{W}_{\alpha,\gamma ,q}(\mathcal{P}),
\end{eqnarray*}
where
\begin{eqnarray*}
\mathcal{W}_{\alpha,\gamma ,q}(\mathcal{P})=\left(\sum_{\phi \neq v\subset S}\left( \gamma_{v}\sum_{\mathbf{k}_{v}\in\mathcal{P}_{v}^\bot} 2^{-\mu '_{\alpha}(\mathbf{k}_{v})} \right) ^q\right) ^{1/q},
\end{eqnarray*}
\begin{eqnarray}
\label{y_norm}
\| f\|_{\mathcal{B}_{\alpha},\gamma ,p,q'}=\left( \sum  _{\phi\neq v\subset S} \left(  \gamma _{v}^{-1}2^{\frac{|v|}{p}}\sup_{\alpha_v\in \{1,\dots ,\alpha \} ^{|v|}}\|f^{(\alpha_v)}\|_{p}\right) ^{q'} \right) ^{1/q'},
\end{eqnarray}
and where
\begin{eqnarray*}
\| f^{(\alpha _v)}\|_p=\left(\int_{[0,1)^{|v|}}\left|\int_{[0,1)^{|S\backslash v|}}\frac{\partial ^{|\alpha _v|}f}{\partial\mathbf{x}_{\alpha _v}}\, d\mathbf{x}_{S\backslash v}\right| ^p d\mathbf{x}_v\right) ^{1/p},
\end{eqnarray*}
with the obvious modifications if either $p,q$ or $q'$ is infinite.
\end{theorem}
In this theorem
we write
\begin{eqnarray*}
\mu _{\alpha}'((l_1,\dots,l_{|v|}))=\sum_{i=1}^{|v|}\sum_{j=1}^{\min (\alpha,N_i)}(a_{i,j}+2)
\end{eqnarray*}
for $l_i=\sum_{j=1}^{N_i}2^{a_{i,j}}$
instead of Dick's weight function $\mu_{\alpha}((l_1,\dots,l_{|v|}))$.

This result yields a significant improvement of (\ref{KH_dick}).
This is crucial when using the bound in a CBC algorithm,
since a large constant (as it appears in \cite[Theorem 3.5]{Dick_garkin})
may make it impractical to perform the CBC construction in practice.
 For instance, \cite[Section~4.1]{GS14} write that
 {\it The resulting large values of the worst-case error bounds}
 [ referring to the large constants in \cite[Theorem~3.5]{Dick_garkin} ] 
{\it have been found to lead to generating vectors with bad projections.}

Additionally, we also include the case $\alpha = \infty$ which has not been studied before in the context of digital nets.
In \cite[Theorem 3.5]{Dick_garkin}, the case $\alpha = \infty$ is not included since in this case the constant $C_\alpha$ appearing in (\ref{KH_dick}) is infinite.
Furthermore, we can define another version of WAFOM when we consider this new bound (see \cite{Harase} for details).

This theorem is based on the estimation of $\mathrm{Err}(f;\mathcal{P})$ by the Walsh coefficients.
Dyadic Walsh coefficients are defined as follows (see \cite{Fine},\cite{SWS} for details).
\begin{definition}[Walsh functions and Walsh coefficients]
Let $f\colon[0,1)^s\rightarrow \mathbb{R}$ and $\mathbf{k}=(k_1,\dots ,k_s)\in\mathbb{N}_0^s$.
We define the $\mathbf{k}$-th dyadic Walsh function $\mathrm{wal}_{\mathbf{k}}$ by
\begin{eqnarray*}
\mathrm{wal}_{\mathbf{k}}(\mathbf{x}):=\prod_{i=1}^s(-1)^{(\sum_{j\ge 1}a _{i,j}b _{i,j})},
\end{eqnarray*}
where for $1\le i\le s$, we write the dyadic expansion of $k_i$ by $k_i=\sum_{j\ge 1}a_{i,j} 2^{j-1}$
and $x_i$ by $x_i=\sum_{j\ge 1}b_{i,j}2^{-j}$,
 where for each $i$, infinitely many digits $b_{i,j}$ are $0$.

Using Walsh functions, we define the $\mathbf{k}$-th dyadic Walsh coefficient $\hat{f}(\mathbf{k})$ as follows:
\begin{eqnarray*}
\hat{f}(\mathbf{k})
:=\int_{[0,1)^s}f(\mathbf{x})\cdot\mathrm{wal}_{\mathbf{k}}(\mathbf{x})\, d\mathbf{x}.
\end{eqnarray*}
\end{definition}

We see that the integration error $\mathrm{Err}(f;\mathcal{P})$ by a digital net $\mathcal{P}$ can be represented by Walsh coefficients $\hat{f}(\mathbf{k})$ as follows
(\cite[Chapter 15]{Dick_Pill}):
\begin{eqnarray*}
\mathrm{Err}(f;\mathcal{P})=
\sum_{\mathbf{k}\in\mathcal{P}^\bot\backslash \{ \mathbf{0}\}}\hat{f}(\mathbf{k}) \qquad
\Big(
 =\sum_{\phi\neq v\subset S}\sum_{\mathbf{k}_v\in\mathcal{P}_v^\bot}\hat{f}(\mathbf{k}_v;\mathbf{0})
\Big) .
\end{eqnarray*}
The proof of Theorem \ref{yKH} is facilitated by an improved bound on the Walsh coefficients of smooth functions.
We show bounds on Walsh coefficients $\hat{f}(\mathbf{k})$ as follows:
\begin{theorem}
\label{yhat}
We assume the same assumptions as in Theorem \ref{yKH}. 
Let $\phi\neq v\subset\{1,\dots ,s\}=S$.
For $\mathbf{k}_v\in\mathbb{N}^{|v|}$, we have
\begin{eqnarray}
\label{yhatineq}
|\hat{f}((\mathbf{k}_v;\mathbf{0}))|
\leq 2^{\frac{|v|}{p}}\cdot 2^{-\mu ' _{\alpha}(\mathbf{k}_v)}
\cdot \| f^{( \min (\alpha,\mathbf{N}_{\mathbf{k}_v}))}\| _p,
\end{eqnarray} 
where $1\le p\le \infty$ and $\| \cdot \|_p$ is the norm defined in Theorem \ref{yKH}.
Here we define the symbol $\min(\alpha,\mathbf{N}_{\mathbf{l}})=(\min(\alpha,N_1),\dots,\min(\alpha,N_{|v|}))$
for $\mathbf{l}=(l_1,\dots,l_{|v|})$ with dyadic expansion $l_i=\sum_{j=1}^{N_i}2^{a_{i,j}}$.
\end{theorem}
This inequality follows from the formula for the Walsh coefficients by dyadic differences, 
which are defined in Section \ref{proofofthm2}
 (see the rough sketch of the proof in Section \ref{proofofthm2}).

Here we compare this result with \cite[Theorem 14]{Dick_decay} and its higher dimensional analogue
 in \cite{Dick_garkin}.
Note that our bound includes the case $\alpha=\infty$ for the case $b=2$
and we see that our bound (\ref{yhatineq}) is better under some condition.
Assume that $s=1$. Then for $N_1 \ge \alpha$, 
if we multiply our bound by $(5/3)^{\alpha-2}$
our bound is still smaller than that bound by Dick for any $k_1=\sum_{j=1}^{N_1}2^{a_{1,j}}$ (see \cite[chapter~14]{Dick_Pill}). 
If $N_1 < \alpha$ it is in general not clear which bound is better.

In the following, we assume that $k\in\mathbb{N}_0$ has dyadic expansion with 
$k=\sum_{j=1}^{N}2^{a_{j}}$ where $N$ is some integer and $a_1>\cdots>a_{N}$ and set $N=0,\{a_j\}=\phi$ for $k=0$.

The remainder of this paper is organized as follows.
In Section \ref{proofofthm1}, we give the proof of Theorem \ref{yKH}
and we give the rough sketch of the proof of Theorem \ref{yhat} in Section \ref{proofofthm2}.
In Section \ref{proofofhatdisc}
we show the proof of lemmas to complete the proof of Theorem \ref{yKH} and \ref{yhat}.
\section{ Proof of Theorem \ref{yKH} }
\label{proofofthm1}
\begin{proof}
We assume that $f$ is continuous on $[0,1]^s$ and
$\sum_{\mathbf{k}\in\mathbb{N}_0^s}|\hat{f}(\mathbf{k})|< \infty$. 
In \cite[Lemma 17]{GSY}, we have pointwise absolute convergence
\begin{eqnarray}
\label{hat_eq}
f(\mathbf{x})=\sum_{\mathbf{k}\in\mathbb{N}_0^s}\hat{f}(\mathbf{k})\mathrm{wal}_{\mathbf{k}}(\mathbf{x}).
\end{eqnarray}
Now we have that $f$ is continuous by the assumption of $f$.
We show that $f$ also satisfies the second condition.
If we apply Theorem \ref{yhat} for $\alpha=2$, we have
\begin{eqnarray*}
\sum_{\mathbf{k}\in\mathbb{N}_0^s\backslash \{ \mathbf{0}\}}|\hat{f}(\mathbf{k})|
= \sum_{\phi\neq v\subset S}\sum_{\mathbf{k}_v\in\mathbb{N}^{|v|}}|\hat{f}((\mathbf{k}_v;\mathbf{0}))|
\le \sum_{\phi\neq  v\subset S}\sum_{\mathbf{k}_v\in\mathbb{N}^{|v|}}2^{\frac{|v|}{p}}2^{-\mu_2 '(\mathbf{k}_v)}\| f^{(\min (2,\mathbf{N}_{\mathbf{k}_v}))}\|_p \\
\le \sum_{\phi\neq  v\subset S}\sum_{\mathbf{k}_v\in\mathbb{N}^{|v|}}2^{\frac{|v|}{p}}2^{-\mu_2 '(\mathbf{k}_v)}\max_{\mathbf{n}\in \{ 1,2\} ^{|v|}}\| f^{(\mathbf{n})}\|_p 
\le 2^{\frac{s}{p}}\max_{\mathbf{n}'\in \{ 0,1,2\} ^s}\| f^{(\mathbf{n}')}\|_{L^p}
\sum_{\phi\neq  v\subset S}\sum_{\mathbf{k}_v\in\mathbb{N}^{|v|}}2^{-\mu_2 '(\mathbf{k}_v)}.
\end{eqnarray*}
Note that $L^p$-norm $\| f\|_{L^p}:=(\int_{[0,1)^s} |f(x)|^pdx)^{1/p}$ is different from 
the norm $\| f\|_p$ defined in Section 1.
Thus we have
\begin{eqnarray*}
\sum_{\mathbf{k}\in\mathbb{N}_0^s}|\hat{f}(\mathbf{k})|
&&= |\hat{f}(\mathbf{0})|+\sum_{\mathbf{k}\in\mathbb{N}_0^s\backslash \{ \mathbf{0}\}}|\hat{f}(\mathbf{k})|\\
&&\le 2^{\frac{s}{p}}\max_{\mathbf{n}'\in \{ 0,1,2\} ^s}\| f^{(\mathbf{n}')}\|_{L^p}
\Big(
1+\sum_{\phi\neq  v\subset S}\sum_{\mathbf{k}_v\in\mathbb{N}^{|v|}}2^{-\mu_2 '(\mathbf{k}_v)}
\Big).
\end{eqnarray*}
Since $\max_{\mathbf{n}'\in\{0,1,2\}^s}\| f^{(\mathbf{n}')}\| _{L^p}<\infty$ holds by the assumption on $f$,
we have only to show the last summation is finite. We prove this in the following way:
\begin{eqnarray*}
1+\sum_{\phi\neq  v\subset S}\sum_{\mathbf{k}_v\in\mathbb{N}^{|v|}}2^{-\mu_2 '(\mathbf{k}_v)}
&&=\left( 1+\sum_{k\in\mathbb{N} }2^{-\mu_{2}'(k)}\right)^s\\
&&=\left( 1+\sum_{l\in\mathbb{N}_0}2^{-l-2}+\sum_{l_1,l_2\in\mathbb{N}_0,l_1<l_2}\sum_{k\in\mathbb{N}_0,k<2^{l_1}}2^{-l_1-l_2-4}\right)^s\\
&&=\left( \frac{3}{2}+\sum_{l_2\in\mathbb{N}_0}l_22^{-l_2-4}\right)^s
\le\left( \frac{3}{2}+ 2^{-4}\cdot \sum_{l_2\in\mathbb{N}_0}\left(\frac{3}{4}\right)^{l_2}\right)^s < \infty.
\end{eqnarray*}
Then we can apply the formula (\ref{hat_eq}) to $f$ to get
\begin{eqnarray*}
|\mathrm{Err}(f;\mathcal{P})|
=\left|\int_{[0,1)^s}f(\mathbf{x}) \ d\mathbf{x}-\frac{1}{|\mathcal{P}|}\sum_{\mathbf{x}\in\mathcal{P}}f(\mathbf{x})\right|
=\left|\hat{f}(\mathbf{0})-\frac{1}{|\mathcal{P}|}\sum_{\mathbf{x}\in\mathcal{P}}\sum_{\mathbf{k}\in\mathbb{N}_0^s}\hat{f}(\mathbf{k})\mathrm{wal}_{\mathbf{k}}(\mathbf{x})\right| .
\end{eqnarray*}
Now we introduce the property of Walsh coefficients $\mathrm{wal}_{\mathbf{k}}$.
Let $\mathcal{P}$ be a digital net in $[0,1)^s$ where $|\mathcal{P}|=2^m$. Then we have
(see \cite[Lemma 4.75]{Dick_Pill})
\begin{eqnarray*}
\sum_{\mathbf{x}\in\mathcal{P}}\mathrm{wal}_{\mathbf{k}}(\mathbf{x})=
\left\{
\begin{array}{ll}
2^m & \mathrm{if} \ \mathbf{k}\in \mathcal{P}^\bot ,\\
0   & \mathrm{otherwise}.\\
\end{array}
\right.
\end{eqnarray*}
Using this fact, we have
\begin{eqnarray*}
|\mathrm{Err}(f;\mathcal{P})|
=\left|\hat{f}(\mathbf{0})-\sum_{\mathbf{x}\in\mathcal{P}^\bot}
\hat{f}(\mathbf{k})\right|
\le \sum_{\mathbf{x}\in\mathcal{P}^\bot\backslash \{ \mathbf{0}\}}|\hat{f}(\mathbf{k})|
=\sum_{\phi\neq v\subset S}\sum_{\mathbf{k}_v\in\mathcal{P}_v^\bot}|\hat{f}((\mathbf{k}_v;\mathbf{0}))|.
\end{eqnarray*}
Let $\alpha\in\mathbb{N}\cup\{\infty\}$ with $\alpha\ge 2$ and $1\le p, q, q'\le \infty$ such that $1/q +1/q' = 1$.
Applying Theorem \ref{yhat} to $f$, we have
\begin{eqnarray*}
|\mathrm{Err}(f;\mathcal{P})|
& &\le  
\sum_{v}\sum_{\mathbf{k}_v\in\mathcal{P}_v^\bot}
2^{\frac{|v|}{p}}2^{-\mu_{\alpha}'(\mathbf{k}_v)}\|f^{(\min(\alpha, \mathbf{N}_{\mathbf{k}_v}))}\|_p\\
& &\le 
\sum_{v}\gamma _{v}^{-1}2^{\frac{|v|}{p}}\sup_{\alpha_v\in\{1,\dots ,\alpha \} ^{|v|}}\|f^{(\alpha_v)}\|_{p}
\sum_{\mathbf{k}_v\in\mathcal{P}_v^\bot}
\gamma _{v}2^{-\mu_{\alpha}'(\mathbf{k}_v)}\\
& &\le \| f\|_{\mathcal{B}_{\alpha},\gamma ,p,q'}\times \mathcal{W}_{\alpha,\gamma ,q}(\mathcal{P}).
\end{eqnarray*}
We use H\"{o}lder's inequality in the last inequality.
\qquad \end{proof}
\section{Proof of Theorem 2}
\label{proofofthm2}
\subsection{Some notations}
To calculate bounds on Walsh coefficients, we introduce the dyadic difference $\partial _{i,n}$
and the weight function $\mu_{\mathbf{u}}'(\mathbf{k})$.
\begin{definition}[dyadic difference]
\label{dfn_dyadic_diff}
Let $s,n,i\in\mathbb{N}$ with $i\leq s$.
For a function $g\colon[0,1)^s\rightarrow \mathbb{R}$,
we define the dyadic difference $\partial _{i,n}(g)$ by
\begin{eqnarray*}
\partial _{i,n} (g)(x_1,\dots ,x_s):=\frac{g(x_1,\dots,x_i\oplus 2^{-n},\dots, x_s )-g(x_1,\dots ,x_i,\dots, x_s)}{2^{-n}}.
\end{eqnarray*}
Here we write $z\oplus 2^{-n}:=z+2^{-n}(-1)^{z_n}$ 
for $z$ having dyadic expansion $z=\sum_{j=1}^{\infty}z_j2^{-j}$, 
where infinitely many digits $z_j$ are $0$.

Let a vector $\mathbf{k}=(k_1,\dots,k_s)\in\mathbb{N}_0^s$, 
$\mathbf{u}=(u_1,\dots,u_s)\in (\mathbb{N}_0\cup\{ \infty\})^s$
with $k_i=\sum_{j=1}^{N_i}2^{a_{i,j}}$.
Let $\phi\neq v\subset S=\{1,\dots,s\}$ and a vector $\mathbf{k}_v\in \mathbb{N}^{|v|}$
with $\mathbf{k}=(\mathbf{k}_v;\mathbf{0})$.
The symbol $(\mathbf{k}_v;\mathbf{0})$ is the same symbol as in Section 1.
We use the symbol
$\mathbf{d}_{\mathbf{k}_{\mathbf{u}}}$
instead of the composition map of 
$\{\partial _{i,a_{i,j}+1}\}_{i\in v,1\le j\le \min(N_i,u_i)}$.
\end{definition}
\begin{remark}{\rm
Since any two dyadic differences commute, $\mathbf{d}_{\mathbf{k}_{\mathbf{u}}}$ is defined independent of the order of a composition.
}\end{remark}

\begin{definition}[the new weight function $\mu ' _{\mathbf{u}}(\mathbf{k})$]
We use the same symbols as in Definition \ref{dfn_dyadic_diff}.
The weight function $\mu ' _{\mathbf{u}}(\mathbf{k})$ of $\mathbf{k}$ is defined by
\begin{eqnarray*}
\mu ' _{\mathbf{u}}(\mathbf{k}):=\sum_{i\in v,1\le j\leq \min\{ N_i,u_i\}}(a_{i,j}+2),
\end{eqnarray*}
and we define $\mu ' _{\mathbf{u}}(\mathbf{0})=0$.

When $u_i=\alpha\in\mathbb{N}\cup\{\infty\}$ for every $i$,
$\mu ' _{\mathbf{u}}((\mathbf{k}_v;\mathbf{0}))$ equals $\mu _{\alpha}'(\mathbf{k}_v)$.
\end{definition}

We define the important two functions $\chi _n(x,y)$ and $W(\mathbf{k})$.
\begin{definition}
\label{W_def}
We define the function $\chi _n(x,y)\colon[0,1)^2\rightarrow\mathbb{R}$ by
\begin{eqnarray*}
\chi _n(x,y):=
\left\{
\begin{array}{ll}
2^n & \mathrm{if} \ y\in [ \min (x,x\oplus 2^{-n}),\max (x,x\oplus 2^{-n})],\\
0 & \mathrm{otherwise}.\\
\end{array}
\right.
\end{eqnarray*}
Recall that $x\oplus 2^{-n}$ is defined as in Definition \ref{dfn_dyadic_diff}.

Using this, we define the $1$-dimensional versions $W(k)$ inductively by
\begin{eqnarray*}
W(0)&&:=1,\\
W(2^{n_1})(y)&&:=\int_0 ^1\chi _{n_1+1}(x,y) \ dx,\\
W(2^{n_1}+\cdots+2^{n_{N+1}})(y)&&:=\int_0 ^1\chi _{n_{N+1}+1}(x,y)W(2^{n_1}+\cdots+2^{n_N})(x) \ dx,
\end{eqnarray*}
where $n_1>\cdots>n_{N+1}$.

Then, for a vector $(k_1,\dots,k_s)\in\mathbb{N}_0^s$,
the $s$-dimensional versions $W(\mathbf{k})\colon[0,1)^s\rightarrow\mathbb{R}$ 
is defined by
\begin{eqnarray*}
W(\mathbf{k}):=\prod_{i=1}^s W(k_i).
\end{eqnarray*}
\end{definition}
\begin{remark}{\rm
By definition, $W(k)$ is continuous on $[0,1)$ for any $k\in \mathbb{N}_0$.
}\end{remark}

The following important property of these functions is proven in Section \ref{section_proof_W}.
\begin{lemma}
\label{prop_{L^1}_{L^1}}
Let a vector $k\in\mathbb{N}$.
We have that $W(k)\geq 0$ on $[0,1)$ and,
for $1\leq p\leq \infty$, we have
\begin{eqnarray*}
\| W(k)\| _{L^p}\le 2^{1-\frac{1}{p}}.
\end{eqnarray*}
\end{lemma}
\subsection{Proof of Theorem \ref{yhat}}
In this subsection, we show bounds on the Walsh coefficients $\hat{f}(\mathbf{k})$
 admitting the following Lemmas \ref{disc_id} and \ref{conti_id},
which we prove in the next section.

We denote the symbols used in the statements.
\begin{definition}
We use the same symbols as in Definition \ref{dfn_dyadic_diff}.
For $\mathbf{k}$ and $\mathbf{u}$, we define
\begin{eqnarray*}
k_{i,>}^{u_i}:=
\left\{
\begin{array}{ll}
\sum_{j>u_i}2^{a_{i,j}}& \quad if \ i\in v ,\\
0&\quad if \ i\in S\backslash v,\\
\end{array}
\right.
k_{i,\le}^{u_i}:=
\left\{
\begin{array}{ll}
\sum_{j\le u_i}2^{a_{i,j}}& \quad if \ i\in v ,\\
0&\quad if \ i\in S\backslash v,\\
\end{array}
\right.
\end{eqnarray*}
and
\begin{eqnarray*}
&&\mathbf{k}^{\mathbf{u}}_{>}:=(k_{i,>}^{u_i})_{i\in S},\qquad
\mathbf{k}^{\mathbf{u}}_{\leq}:=(k_{i,\le}^{u_i})_{i\in S},\\
&&{\mathbf{k}_v}^{\mathbf{u}}_{>}:=(k_{i,>}^{u_i})_{i\in v},\qquad
{\mathbf{k}_v}^{\mathbf{u}}_{\leq}:=(k_{i,\le}^{u_i})_{i\in v},\\
&&\min(\mathbf{u},\mathbf{N}_{\mathbf{k}}):=(\min(u_i,N_i))_{i\in S},\qquad
\min(\mathbf{u},\mathbf{N}_{\mathbf{k}_v}):=(\min(u_i,N_i))_{i\in v},\\
&&|\min(\mathbf{u},\mathbf{N}_{\mathbf{k}})|_{l^1}:=\sum_{i\in S}\min(u_i,N_i) .
\end{eqnarray*}
\end{definition}

When we analyze Walsh coefficients, it is suitable to use dyadic differences. In fact, 
the $\mathbf{k}$-th Walsh coefficient $\hat{f}(\mathbf{k})$ can be represented by
$\widehat{\partial _{i,n}f}(\mathbf{k})$ as follows:
\begin{eqnarray}
\label{disc_id_ex}
\hat{f}(\mathbf{k})=(-1)\cdot 2^{-a_{i,j}-2}\widehat{\partial _{i,a_{i,j}+1}(f)}(\mathbf{k}),
\end{eqnarray}
where $\mathbf{k}=(k_1\dots,k_s)$ with $k_i=\sum_{j=1}^{N_i}2^{a_{i,j}}$.
Applying the above formula (\ref{disc_id_ex}) repeatedly, we have the following formula:
\begin{lemma}
\label{disc_id}
Let $f\in L^1([0,1)^s)$, $\mathbf{k}\in\mathbb{N}_0^s$
and $\mathbf{u}\in(\mathbb{N}_0\cup\{\infty\})^s$.\\
Then we have
$\mathbf{d}_{\mathbf{k}_{\mathbf{u}}}f\in L^1([0,1)^s)$
and we have
\begin{eqnarray}
\label{disc_id_ex2}
\hat{f}(\mathbf{k})=(-1)^{|\min( \mathbf{u},\mathbf{N}_{\mathbf{k}} )|_{l^1}}
2^{-\mu '_{\mathbf{u}} (\mathbf{k})}\widehat{\mathbf{d}_{\mathbf{k}_{\mathbf{u}}}f}(\mathbf{k}).
\end{eqnarray}
\end{lemma}
This formula (\ref{disc_id_ex2}) means that dyadic differences
 connect the $\mathbf{k}$-th Walsh coefficient $\hat{f}(\mathbf{k})$
 to the weight function $\mu'_{\mathbf{u}}(\mathbf{k})$ for $f\in L^1([0,1)^s)$. 

Dyadic differences $\partial _{i,n}f$ are similar to derivatives $\frac{\partial f}{\partial x_i }$.
In the following formula, we can replace $\partial _{i,n}f$ with $\frac{\partial f}{\partial x_i }$ using a slight modification.

\begin{lemma}
\label{conti_id}
Let $\mathbf{u}=(u_1,\dots,u_s) \in(\mathbb{N}_0\cup\{\infty\})^s.$
We assume that a function $f$ satisfies that its mixed partial
derivatives up to order $u_i$ in each variable $x_i$ are continuous on $[0,1]^s.$
For any vector $\mathbf{k}\in\mathbb{N}_0^s,$ we have
\begin{eqnarray*}
\hat{f}(\mathbf{k})& &=(-1)^{|\min(\mathbf{u},\mathbf{N}_{\mathbf{k}})|_{l^1}}
2^{-\mu ' _{\mathbf{u}}(\mathbf{k})}
\int_{[0,1)^s}f^{(\min(\mathbf{u},\mathbf{N}_{\mathbf{k}}))}(\mathbf{x}) \cdot
W(\mathbf{k}_{\leq}^{\mathbf{u}})(\mathbf{x})
\cdot \mathrm{wal}_{\mathbf{k}^\mathbf{u}_{>}}(\mathbf{x})\ d\mathbf{x}.
\end{eqnarray*}
\end{lemma}
Then we can get the following bound on $|\hat{f}(\mathbf{k})|$:
\begin{lemma}
\label{conti_ineq}
We assume the same assumptions as in Lemma \ref{conti_id}.
Let $\phi\neq v\subset \{ 1,\dots, s\}=S$ and $\mathbf{k}_v\in\mathbb{N}^{|v|}$.
Then we have
\begin{eqnarray*}
|\hat{f}((\mathbf{k}_v;\mathbf{0}))|
\leq 2^{\frac{|v|}{p}}\cdot 2^{-\mu ' _{\mathbf{u}}((\mathbf{k}_v;\mathbf{0}))}
\cdot \| f^{(\min(\mathbf{u},\mathbf{N}_{\mathbf{k}_v}))}\| _p <\infty,
\end{eqnarray*} 
where $1\le p\le \infty$ and $\| \cdot \|_p$ is the norm defined in Theorem \ref{yKH}.
\end{lemma}
\begin{proof}
We write $\mathbf{x}_v=(x_i)_{i\in v}$ for $\mathbf{x}\in[0,1)^s$.
We see that $W((\mathbf{k}_v;\mathbf{0})_{\leq}^{\mathbf{u}})(\mathbf{x})\equiv W({\mathbf{k}_v}_{\leq}^{\mathbf{u}})(\mathbf{x}_v)$
and $\mathrm{wal}_{(\mathbf{k}_v;\mathbf{0})^\mathbf{u}_{>}}(\mathbf{x})\equiv \mathrm{wal}_{{\mathbf{k}_v}^\mathbf{u}_{>}}(\mathbf{x}_v)$
since $W(0)\equiv\mathrm{wal}_{0}\equiv 1$.
And we see $|\mathrm{wal}_{{\mathbf{k}_v}^\mathbf{u}_{>}}|\equiv 1$
by the definition of Walsh functions.
Combining those facts and Lemma \ref{conti_id}, we have
\begin{eqnarray*}
& &|\hat{f}((\mathbf{k}_v;\mathbf{0}))|\le 2^{-\mu ' _{\mathbf{u}}((\mathbf{k}_v;\mathbf{0}))}
\left|\int_{[0,1)^s}f^{((\min (\mathbf{u},\mathbf{N}_{\mathbf{k}_v});\mathbf{0}))}(\mathbf{x}) \cdot W({\mathbf{k}_v}_{\leq}^{\mathbf{u}})(\mathbf{x})\cdot \mathrm{wal}_{{\mathbf{k}_v}^\mathbf{u}_{>}}(\mathbf{x})\, d\mathbf{x}\right|\\
& &\le 2^{-\mu ' _{\mathbf{u}}((\mathbf{k}_v;\mathbf{0}))}
 \int_{[0,1)^v}\left|\int_{[0,1)^{|S\backslash v|}}f^{( (\min(\mathbf{u},\mathbf{N}_{\mathbf{k}_v});\mathbf{0}))}(\mathbf{x})\ d\mathbf{x}_{S\backslash v}\right|
 \cdot  \left| W({\mathbf{k}_v}_{\leq}^{\mathbf{u}})(\mathbf{x}_v)\right|\, d\mathbf{x}_v\\
& &\le 2^{-\mu ' _{\mathbf{u}}((\mathbf{k}_v;\mathbf{0}))}
\|f^{(\min(\mathbf{u},\mathbf{N}_{\mathbf{k}_v}))}\|_p
\cdot \| W({\mathbf{k}_v}_{\leq}^{\mathbf{u}})\|_{L^{\frac{p}{p-1}}},
\end{eqnarray*}
where we used H\"{o}lder's inequality in the third inequality.
By Lemma \ref{prop_{L^1}_{L^1}}, we have
\begin{eqnarray*}
\| W({\mathbf{k}_v}_{\leq}^{\mathbf{u}})\|_{L^{\frac{p}{p-1}}}
=\prod_{i\in v}\| W({k_i}_{\leq}^{u_i})\|_{L^{\frac{p}{p-1}}}
\le 2^{\frac{|v|}{p}}.
\end{eqnarray*}
Thus we obtain the result.
\qquad \end{proof}

In particular, when $u_i=\alpha\in\mathbb{N}\cup\{\infty\}$ for every $i$, we have Theorem \ref{yhat}.

In the following  section, we will prove the lemmas which we used in this section.
From now, we denote by $\prod_{i=1}^n\varphi_i$
the composition of maps $\varphi _1\circ\dots\circ\varphi _n$.  

\section{Proof of Lemmas}
\label{proofofhatdisc}
\begin{definition}
\label{dpq_def}
We use the same symbols as in Definition \ref{dfn_dyadic_diff}.
Let $\mathbf{p}=(p_i)_{i\in v},\mathbf{q}=(q_i)_{i\in v}\in\mathbb{N}^{|v|}$ with $1\le p_i\le q_i\le N_i$
We use the following symbols in the proof.
\begin{eqnarray*}
\mathbf{d}_{\mathbf{p}}^{\mathbf{q}}:=
\prod_{i\in v,p_i\le j\le q_i}\partial_{i,a_{i,j}+1}.
\end{eqnarray*}
We use this symbol when we can recognize $\mathbf{k}$. 
\end{definition}

\subsection{Proof of Lemma \ref{disc_id} }
\begin{proof}
We prove only the case $s=1$ here.
In the case $s>1$, we obtain the result by applying the same method in a component-wise fashion.

We easily obtain the first statement as follows.
Let $k_1=\sum_{j=1}^{N_1}2^{a_{j}}$.
Since $\partial_{1,a_{j}+1}$ is the sum of $f(x_1\oplus 2^{-a_{j}-1})\in L^1([0,1))$ and $f\in L^1([0,1))$,
we have $\partial _{1,a_{j}+1}f\in L^1([0,1)^s)$.
By repeating this argument, we have $\mathbf{d}_{{(k_1)}_{(u_1)}}(f)\in L^1([0,1))$.

We show the second statement inductively.
We omit the case $k_1=0$
or $u_1=0$ since the proof is easy.
We show the case $u_1=1$:
\begin{eqnarray}
\label{hat_id_lem_s1}
\hat{f}(k_1)=(-1)\cdot 2^{-a_{j}-2}\cdot\widehat{\partial _{1,a_{j}+1}(f)}(k_1),
\end{eqnarray}
where $k_1=\sum_{j=1}^{N_1}2^{a_{j}}$.
By changing variables $x_1\mapsto x_1\oplus2^{-a_{j}-1}$, we have
\begin{eqnarray*}
& &\int_0^1 f(x_1\oplus 2^{-a_{j}-1}) \cdot \mathrm{wal}_{k_1}(x_1) \ dx_1\\
& &=\sum_{c=0}^{2^{a_{j}}-1}
\Big( \int_{2^{-a_{j}-1}\cdot 2c}^{2^{-a_{j}-1}\cdot (2c+1)}
 f(x_1+2^{-a_{j}-1}) \cdot \mathrm{wal}_{k_1}(x_1) \ dx_1\\
& &+\int_{2^{-a_{j}-1}\cdot (2c+1)}^{2^{-a_{j}-1}\cdot (2c+2)}
 f(x_1- 2^{-a_{j}-1}) \cdot \mathrm{wal}_{k_1}(x_1) \ dx_1\Big)\\
& &=\sum_{c=0}^{2^{a_{j}}-1}\Big(\int_{2^{-a_{j}-1}\cdot (2c+1)}^{2^{-a_{j}-1}\cdot (2c+2)}
 f(x_1) \cdot \mathrm{wal}_{k_1}(x_1- 2^{-a_{j}-1}) \ dx_1\\
& &+\int_{2^{-a_{j}-1}\cdot 2c}^{2^{-a_{j}-1}\cdot (2c+1)}
 f(x_1) \cdot \mathrm{wal}_{k_1}(x_1+ 2^{-a_{j}-1}) \ dx_1\Big)\\
& &=\int_0^1 f(x_1) \cdot \mathrm{wal}_{k_1}(x_1\oplus 2^{-a_{j}-1}) \ dx_1
=\int_0^1 f(x_1) \cdot \mathrm{wal}_{k_1}(x_1)\cdot
\mathrm{wal}_{k_1}(2^{-a_{j}-1}) \ dx_1\\
& &=-\int_0^1 f(x_1) \cdot \mathrm{wal}_{k_1}(x_1) \ dx_1,
\end{eqnarray*}
where the last two identities follow from the definition of Walsh functions.
Using this calculation, we obtain
\begin{eqnarray*}
\widehat{\partial _{1,a_{j}+1}(f)}\ (k_1)=
-2\cdot 2^{a_{j}+1}\cdot\int_0^1 f(x_1) \cdot\mathrm{wal}_{k_1}(x_1) \ dx_1=(-1)\cdot 2^{a_{j}+2}\cdot\hat{f}(k_1).
\end{eqnarray*}
We write $U=\min(u_1,N_1)$.
Using (\ref{hat_id_lem_s1}) inductively, we obtain
\begin{eqnarray*}
& &\widehat{\mathbf{d}_{{(k_1)}_{(u_1)}}f}\ (k_1)
=\widehat{\mathbf{d}_{(1)}^{(U)}f}\ (k_1)
=(-1)\cdot 2^{a_{1}+2}\cdot\widehat{\mathbf{d}_{(2)}^{(U)}f}\ (k_1)\\
& &=(-1)^2\cdot 2^{\sum_{j=1}^2(a_{j}+2)}\cdot\widehat{\mathbf{d}_{(3)}^{(U)}f}\ (k_1)
=\dots
=(-1)^{U}\cdot 2^{\mu'_{(u_1)}((k_1))}\cdot\hat{f}(k_1),
\end{eqnarray*}
which is the result.
\qquad \end{proof}
\subsection{Proof of Lemma \ref{conti_id}}
\label{proofofhatconti}
\subsubsection{Important properties of dyadic differences $\partial_{i,n}$}
In order to prove Lemma \ref{conti_id}, we show some properties of dyadic differences $\partial_{i,n}$.
We define the following symbols.
\begin{definition}
\label{wd}
We use the same symbols as in Definition \ref{dfn_dyadic_diff} and \ref{dpq_def}.
We define
\begin{eqnarray*}
w_{i,n}&&:=\mathrm{wal}_{2^{n-1}}(x_i),\\
w\partial_{i,n}(g)&&:=w_{i,n}\cdot \partial_{i,n}(g)\quad \mathrm{for} \ g\colon [0,1)^s\rightarrow\mathbb{R},
\end{eqnarray*}
and
\begin{eqnarray*}
\mathbf{w}_{\mathbf{p}}^{\mathbf{q}}:=
\prod_{i\in v,p_i\le j\le q_i}w_{i,a_{i,j}+1},\quad 
\mathbf{wd}_{\mathbf{p}}^{\mathbf{q}}:=
\prod_{i\in v,p_i\le j\le q_i}w\partial_{i,a_{i,j}+1}.
\end{eqnarray*}
Notice that $\mathbf{w}_{\mathbf{p}}^{\mathbf{q}}$ is a function but
$\mathbf{wd}_{\mathbf{p}}^{\mathbf{q}}$ is an operator.
We can rewrite the Walsh function as follows: 
$\mathrm{wal}_{\mathbf{k}}=\mathbf{w}_{(1,\dots, 1)}^{(N_1,\dots ,N_s)}$.
\end{definition}

We see that $w_{j,m}$ and $\partial_{i,n}$ commute in the next lemma.
\begin{lemma}
\label{wd_dw}
When $(i,n)\neq(j,m)\in\mathbb{N}^2$, for a function $g\colon[0,1)^s\rightarrow\mathbb{R}$,
we have the following identity:
\begin{eqnarray*}
w_{j,m}\cdot\partial_{i,n}(g)=\partial_{i,n}(g\cdot w_{j,m}).
\end{eqnarray*}
\end{lemma}
\begin{proof}
We omit the proof here since it is easy.
\qquad \end{proof}

We first prove the following property.
\begin{lemma}
\label{wd_L^1}
We use the same symbols in the above definition. 
For a function $g\in L^1([0,1)^s)$, we have $\mathbf{wd}_{\mathbf{p}}^{\mathbf{q}}g \in L^1([0,1)^s)$. 
\end{lemma}
\begin{proof}
Since the above lemma, we have that
$\mathbf{wd}_{\mathbf{p}}^{\mathbf{q}}g$
equals $\mathbf{w}_{\mathbf{p}}^{\mathbf{q}}\cdot\mathbf{d}_{\mathbf{p}}^{\mathbf{q}}g$.
By the definition of Walsh functions, we see $|\mathbf{w}_{\mathbf{p}}^{\mathbf{q}}|\equiv  1$. 
And since $\mathbf{d}_{\mathbf{p}}^{\mathbf{q}}g$
is the sum of the functions in $L^1([0,1)^s)$ as in the proof of Lemma \ref{disc_id},
we have $\mathbf{d}_{\mathbf{p}}^{\mathbf{q}}g \in L^1([0,1)^s)$. 
Thus the result follows.
\qquad \end{proof}
\subsubsection{Proof of Lemma \ref{conti_id}}
The following Lemma is the key to prove Lemma \ref{conti_id},
which connects $w\partial_{i,n}(g)$ with the derivative $\frac{\partial g}{\partial x_i}$.
\begin{lemma}
\label{chi}
Let $n,s,i\in\mathbb{N}$ satisfy $s\ge i$. Let $g\colon[0,1]^s\to\mathbb{R}$ as a function of 
the $i$th component $x_{i}$, satisfy
\begin{eqnarray}
\label{f_C_{L^1}}
g\in C^1\left( \big[ 2^{-n+1}c,2^{-n+1}(c+1) \big) \right), \quad c=0,\dots,2^{n-1}-1.
\end{eqnarray}
Then for any $\mathbf{z}=(z_1,\dots,z_s)\in[0,1)^s$, we have
\begin{eqnarray*}
w\partial_{i,n}(g)(\mathbf{z})=\int_0^1{\frac{\partial g}{\partial x_{i}}}^*
(z_1,\dots, z_{i-1}, y,z_{i+1}, \dots, z_s )
\cdot \chi_n (z_{i},y) \ d y,
\end{eqnarray*}
where we define 
\begin{eqnarray*}
{\frac{\partial g}{\partial x_i}}^*:=\frac{\partial g}{\partial x_i}
\quad \mathrm{on} \ x_i\in\big[ 2^{-n+1}c,2^{-n+1}(c+1)\big)
, \quad c=0,\dots,2^{n-1}-1.
\end{eqnarray*}
\end{lemma}
\begin{proof}
Let $c'\in\mathbb{N}_0$ satisfying $z_{i}\in [2^{-n}c',2^{-n}(c'+1))$.
We consider two cases: $c'=2c$ and $c'=2c+1$ for some integer $c$.
We only calculate the case $c=2c'$ since the other case can be calculated by the same way.
In this case, by the calculation $w_{i,n}(z_{i})=1$ and the assumption (\ref{f_C_{L^1}}), we have
\begin{eqnarray*}
w\partial_{i,n}(g)(\mathbf{z})
& &=\partial_{i,n}(g)(\mathbf{z})
=2^n\cdot \left( g(z_1,\dots, z_{i}+2^{-n},\dots ,z_s )-g(z_1,\dots ,z_s ) \right) \\
& &=\int_{z_{i}}^{z_{i}+2^{-n}}2^n\cdot \frac{\partial g}{\partial x_{i}}(z_1,\dots ,z_{i-1},y,z_{i+1}, \dots ,z_s) \ d y\\
& &=\int _0^1 {\frac{\partial g}{\partial x_{i}}}^*(z_1,\dots ,z_{i-1},y,z_{i+1},\dots ,z_s) \cdot \chi_n (z_{i},y) \ d y.
\end{eqnarray*}
The last equality follows from $[z_{i},z_{i}+2^{-n}]=\left[\min (z_{i},z_{i}\oplus 2^{-n}),\max (z_{i},z_{i}\oplus 2^{-n})\right]$ and the definition of $\chi_n$. 
\qquad \end{proof}

We prove Lemma \ref{conti_id} using these results.
\begin{proof}
We prove the case $s=1$.
We omit the case $k_1=0$ or $u_1=0$ since the proof is easy.
We write $k_1=\sum_{j=1}^N2^{a_j}$ and $U=\min(u_1,N)$ here.
We assume that $u_1\ge 1$, then we have
\begin{eqnarray*}
\widehat{\mathbf{d}_{{(k_1)}_{(u_1)}} f}\ (k_1)
& &=\widehat{\mathbf{d}_{(1)}^ {(U)} f}\ (k_1)
=\int_0^1\Big(\mathbf{d}_{(1)}^ {(U)}f\Big) ( x_1)
\cdot\mathbf{w}_{(1)}^ {(N)}(x_1) \, dx_1\\
& &=\int_0^1
\mathbf{wd}_{(1)}^ {(U)} \Big(f\cdot\mathbf{w}_ {(U+1)}^ {(N)}\Big) (x_1) \, dx_1.
\end{eqnarray*}
We use Lemma \ref{wd_dw} in the third equality. Using the assumption of $f$ and the definition of $w_{a_j+1}$, we have that
\begin{eqnarray*}
f\cdot\mathbf{w}_ {(U+1)}^ {(N)}
\in C^{u_1}\left( \big[ 2^{-a_ {U+1}-1}c,2^{-a_{U +1}-1}(c+1)\big) \right) ,
\end{eqnarray*}
and we have
\begin{eqnarray*}
\frac{d }{d x_1}\left(f\cdot\mathbf{w}_ {(U+1)}^ {(N)}\right)
=\Big(\frac{d f}{d x_1}\cdot\mathbf{w}_ {(U+1)}^ {(N)}\Big)
\quad \mathrm{on}\ \big[ 2^{-a_ {U+1}-1}c,2^{-a_ {U+1}-1}(c+1)\big) ,
\end{eqnarray*}
with $0\le c\le 2^{a_ {U+1}+1}-1$.

Let $1\le n,0\le c'< 2^n$ be integers.
By the definition of $wd_{1,n}$, we have that, if
$g\in C^{1}\big([c2^{-n},(c+1)2^{-n})\big)$,
it holds that 
$wd_{1,n}g\in C^{1}\big([c2^{-n},(c+1)2^{-n})\big)$ and
$\frac{d}{d x_1} (wd_{1,n}g)=
wd_{1,n} (\frac{d g}{d x_1})$
on $[c2^{-n},(c+1)2^{-n})$.

If we take $g=f\cdot\mathbf{w}_ {(U+1)}^ {(N)}$ and $n=a_ {U}+1$, we have
\begin{eqnarray*}
\frac{d }{d x_1}
\left(wd_{1,a_ {U}+1} \Big(f\cdot\mathbf{w}_ {(U+1)}^ {(N)}\Big)\right)
=wd_{1,a_ {U}+1}
\Big(\frac{d f}{d x_1}\cdot\mathbf{w}_ {(U+1)}^ {(N)}\Big)
\quad \mathrm{on}\ \big[c2^{-a_ {U}-1},(c+1)2^{-a_ {U}-1}\big) ,
\end{eqnarray*}
where $0\le c\le 2^{a_ {U}+1}-1$.
Applying this argument inductively,
we have
\begin{eqnarray*}
\frac{d }{d x_1}
\left(\mathbf{wd}_{(2)}^ {(U)} \Big( f\cdot\mathbf{w}_ {(U+1)}^ {(N)}\Big)\right)
=\mathbf{wd}_{(2)}^ {(U)}
\Big(\frac{d f}{d x_1}\cdot\mathbf{w}_ {(U+1)}^ {(N)}\Big)
\quad \mathrm{on}\ \big[c2^{-a_{2}-1},(c+1)2^{-a_{2}-1}\big) ,
\end{eqnarray*}
where $0\le c\le 2^{a_{2}+1}-1$.
Since $2^{-a_{2}-1}\ge 2^{-a_{1}}$, we can take
$n=a_{1}+1$ and $g=\mathbf{wd}_{(2)}^ {(U)} \Big( f\cdot\mathbf{w}_ {(U+1)}^ {(N)}\Big)$ in Lemma \ref{chi}.
Then we continue the computation of $\widehat{\mathbf{d}_{{(k_1)}_{(u_1)}}f}\ (k_1)$ as follows
\begin{eqnarray*}
& &\widehat{\mathbf{d}_{{(k_1)}_{(u_1)}}f}\ (k_1)
=\int_0^1 wd_{1,a_{1}+1}
\Big(\mathbf{wd}_{(2)}^ {(U)}\big( f\cdot\mathbf{w}_ {(U+1)}^ {(N)}\big)\Big) (x_1) \, dx_1\\
& &=\int_0^1
\left(\int_0^1 \chi_{a_{1}+1}(x_1,y) \cdot
\mathbf{wd}_{(2)}^ {(U)}
\Big(\frac{d f}{d x_1}\cdot\mathbf{w}_ {(U+1)}^ {(N)}\Big) (y) \, dy\right) \, dx_1.
\end{eqnarray*}
Now we have $\frac{d f}{d x_1}\cdot\mathbf{w}_ {(U+1)}^ {(N)}\in L^1([0,1))$ since $|w_{a_{j}+1}|\equiv 1$ and the assumption of $f$.
Therefore if we take $g=\frac{d f}{d x_1}\cdot\mathbf{w}_ {(U+1)}^ {(N)}$
in Lemma \ref{wd_L^1} and consider the fact $|\chi_{a_{1}+1}(x_1,y)|\le 2^{a_{1}+1}$,
we see that the integrand  $\chi_{a_{1}+1} \cdot\mathbf{wd}_{(2)}^ {(U)}
\Big(\frac{d f}{d x_1}\cdot\mathbf{w}_ {(U+1)}^ {(N)}\Big)$ in the last line
belongs to $L^1([0,1)^2)$. Thus we can use Fubini's Theorem as follows.
\begin{eqnarray*}
\widehat{\mathbf{d}_{{(k_1)}_{(u_1)}}f}\ (k_1)
& &=\int_0^1\left(\int_0^1 \chi_{a_{1}+1}(x_1,y) \, dx_1\right)\cdot
\mathbf{wd}_{(2)}^ {(U)}
\Big(\frac{d f}{d x_1}\cdot\mathbf{w}_ {(U+1)}^ {(N)}\Big) (y) \, dy\\
& &=\int_0^1\mathbf{wd}_{(2)}^ {(U)}
\Big(\frac{d f}{d y}\cdot\mathbf{w}_ {(U+1)}^ {(N)}\Big) (y) \cdot W(2^{a_{1}})(y)\, dy\\
& &=\int_0^1\mathbf{wd}_{(2)}^ {(U)}
\Big(\frac{d f}{d x_1}\cdot\mathbf{w}_ {(U+1)}^ {(N)}\Big) (x_1) \cdot W(2^{a_{1}})(x_1)\, dx_1.
\end{eqnarray*}
By repeating the argument we have
\begin{eqnarray*}
\widehat{\mathbf{d}_{{(k_1)}_{(u_1)}}f}\ (k_1)
& &=\int_0^1 wd_{1,a_2+1}\left(\mathbf{wd}_{(3)}^ {(U)}
\Big(\frac{d f}{d x_1}\cdot\mathbf{w}_ {(U+1)}^ {(N)}\Big)\right) (x_1) \cdot W(2^{a_{1}})(x_1)\, dx_1\\
& &=\int_0^1
\left(\int_0^1 \chi_{a_{2}+1}(x_1,y) \cdot
\mathbf{wd}_{(3)}^ {(U)}
\Big(\frac{d f}{d x_1}\cdot\mathbf{w}_ {(U+1)}^ {(N)}\Big) (y) \, dy\right)\cdot
W(2^{a_1})(x_1)\, dx_1\\
& &=\int_0^1\left(\int_0^1 \chi_{ a_{2}+1}(x_1,y)\cdot W(2^{ a_{1}})(x_1) \, dx_1\right)\cdot
\mathbf{wd}_{(3)}^ {(U)}
\Big(\frac{d^2 f}{d x_1^2}\cdot\mathbf{w}_ {(U+1)}^ {(N)}\Big) (y) \, dy\\
& &=\int_0^1\mathbf{wd}_{(3)}^ {(U)}
\Big(\frac{d ^2 f}{d x_1^2}\cdot
\mathbf{w}_ {(U)}^ {(N)}\Big) (y) \cdot W(2^{a_{1}}+2^{a_{2}})(y)\, dy\\
& &=\int_0^1\mathbf{wd}_{(3)}^ {(U)}
\Big(\frac{d ^2 f}{d x_1^2}\cdot
\mathbf{w}_ {(U)}^ {(N)}\Big) (x_1) \cdot W(2^{a_{1}}+2^{a_{2}})(x_1)\, dx_1\\
& &=\cdots \\
& &=\int_0^1
\Big(\frac{d ^ {U}f}{d x_1^ {U}}\cdot\mathbf{w}_ {(U)}^ {(N)}\Big)(x_1)\cdot
W(k_{1,\leq}^ {u_1})(x_1) \ dx_1,
\end{eqnarray*}
thus we have the result for $s=1$.
By calculating in a component-wise manner, we have the result for the case $s\ge 1$.
We omit that case here.
\qquad \end{proof}

In fact, we can determine the sign of $\hat{f}(\mathbf{k})$ in the special case.
\begin{corollary}
Let $f\in C^{\infty}([0,1]^s)$ and $\mathbf{k}\in\mathbb{N}_0^s$.
We use the symbol $N_i$ appearing in Definition \ref{dfn_dyadic_diff}.
Then, if $f^{(N_1,\dots, N_s)}\ge 0$, we have $\hat{f}(\mathbf{k})\cdot (-1)^{\sum_{i=1}^s N_i}\ge 0$.
\end{corollary}
\begin{proof}
By Lemma \ref{prop_{L^1}_{L^1}} and the fact $W(0)\equiv 1$, we have $W(\mathbf{k})=\prod_{i=1}^sW(k_i)\ge 0.$
By combining this fact and the assumption that $f^{(N_1,\dots, N_s)}\ge 0$, 
we have that the product $f^{(N_1,\dots,N_s)}\cdot W(\mathbf{k})\ge 0$.
Thus, by Lemma \ref{conti_id} with $u_i=\infty$ we have
\begin{eqnarray*}
\hat{f}(\mathbf{k})\cdot (-1)^{\sum_{i=1}^s N_i}=2^{-\mu '_{\infty} (\mathbf{k})}\cdot
\int_{[0,1)^s}f^{(N_1,\dots ,N_s)}(\mathbf{x}) \cdot W(\mathbf{k})(\mathbf{x}) \ d\mathbf{x}\ge0,
\end{eqnarray*}
which is the result.
\qquad \end{proof}
\subsection{Proof of Lemma \ref{prop_{L^1}_{L^1}}}
\label{section_proof_W}
\subsubsection{Important properties of $\chi_n(x,y)$ and $W(k)$}
We show the important properties of $\chi_n(x,y)$ and $W(k)$ in this subsection.
See Definition \ref{W_def} for the definitions of $\chi_n(x,y)$ and $W(k)$. 

We see that $\chi_n(x,y)/2^n$ is a characteristic function of some region in $[0,1)^2$.
\begin{lemma}
\label{period_chi}
Let $n\in\mathbb{N}$ and $c\in\mathbb{N}_0$ satisfying $c<2^{n-1}.$\\
\begin{enumerate}
\item
Let $x,y\in[0,2^{-n+1}),$ then we have
\begin{eqnarray*}
\chi_n(x+c2^{-n+1},y+c2^{-n+1})=\chi_n(x,y).
\end{eqnarray*}
\item
Let $x\in [c2^{-n+1},(c+1)2^{-n+1})$ and $y\not\in [c2^{-n+1},(c+1)2^{-n+1})$. Then we have $\chi_n(x,y)=0$.

And let $y\in [c2^{-n+1},(c+1)2^{-n+1})$ and $x\not\in [c2^{-n+1},(c+1)2^{-n+1})$. Then we have $\chi_n(x,y)=0$.
\end{enumerate}
\end{lemma}
\begin{proof} 
\begin{enumerate}
\item
We have $(x+c2^{-n+1})\oplus 2^{-n}=(x\oplus 2^{-n})+c2^{-n+1}.$
Thus, the result follows from the fact that
\begin{eqnarray*}
y&&\in [ \min (x,x\oplus 2^{-n}),\max (x,x\oplus 2^{-n}) ]\\
&& \Longleftrightarrow  \\
y+c 2^{-n+1}&&\in [ \min (x+c 2^{-n+1},(x+c 2^{-n+1})\oplus 2^{-n}),\\
&&\max (x+c 2^{-n+1},(x+c 2^{-n+1})\oplus 2^{-n}) ].
\end{eqnarray*}
\item
We prove $\chi_n(x,y)=0$ in the case $x\in [c2^{-n+1},(c+1)2^{-n+1})$ and $y\not\in [c2^{-n+1},(c+1)2^{-n+1})$.
Let $x\in [d2^{-n},(d+1)2^{-n})$ for $d\in\mathbb{N}_0$ where $d=2c$ or $2c+1$.
When $d=2c$, it holds that $c2^{-n+1}\le x<x\oplus 2^{-n}=x+2^{-n}<(c+1)2^{-n+1}$.
In the case $d=2c+1$, it holds that $c2^{-n+1}\le x\oplus 2^{-n}=x-2^{-n}<x<(c+1)2^{-n+1}.$
So we have
\begin{eqnarray*}
[ \min (x,x\oplus 2^{-n}),\max (x,x\oplus 2^{-n}) ] \subset[c 2^{-n+1},(c+1)2^{-n+1}).
\end{eqnarray*}
So, if $y\not\in [c2^{-n+1},(c+1)2^{-n+1})$, we have $y\not\in [ \min (x,x\oplus 2^{-n}),\max (x,x\oplus 2^{-n}) ]$.
Then we obtain $\chi_n(x,y)=0$.

For the case $x\not\in [c2^{-n+1},(c+1)2^{-n+1})$ and $y\in [c2^{-n+1},(c+1)2^{-n+1})$,
there is some integer $e$ such that $x\in [e2^{-n+1},(e+1)2^{-n+1})$ and $y\not\in [e2^{-n+1},(e+1)2^{-n+1})$.
Thus the result follows from the above argument.
\end{enumerate}
We finish the proof.
\qquad \end{proof}

The function $\chi _n(x,y)$ is defined by using a characteristic function of $y$.   
In Lemma \ref{chi_y}, we rewrite $\chi _n(x,y)$ by using a characteristic function of $x$.
\begin{lemma}
\label{chi_y}
Let $y\in[0,1)$ and $c,n\in\mathbb{N}_0$ satisfy $y\in [2^{-n}c,2^{-n}(c+1))$.

If $c=2c'$ for some integer $c'$, we have 
\begin{eqnarray*}
\chi _n(x,y)=
\left \{
\begin{array}{ll}
2^n & \mathrm{if}\ x\in[2^{-n}c,y]\cup [2^{-n}(c+1),y+2^{-n}],\\
0 & \mathrm{otherwise}.\\
\end{array}
\right.
\end{eqnarray*}
And if $c=2c'+1$ for some integer $c'$, we have 
\begin{eqnarray*}
\chi _n(x,y)=
\left \{
\begin{array}{ll}
2^n & \mathrm{if} \ x\in[y-2^{-n},2^{-n}c)\cup [y,2^{-n}(c+1)),\\
0  & \mathrm{otherwise}.\\
\end{array}
\right.
\end{eqnarray*}
\end{lemma}
\begin{proof}
We only prove the case $c=2c'$ here since the case $c=2c'+1$ follows from the same argument.
In this case, by Item 2 of Lemma \ref{period_chi}, we have that for all $y\in[ 2^{-n+1}c',2^{-n+1}(c'+1) )$,
\begin{equation}
\label{tri_0}
\chi _n(x,y)=0, \quad x\not\in[2^{-n+1}c',2^{-n+1}(c'+1)).
\end{equation}
Let $x\in[0,1)$ and $d\in\mathbb{N}_0$ satisfy $x\in [2^{-n}d,2^{-n}(d+1))$.
We calculate $\chi _n(x,y)$ in the three cases: $d\not\in\{2c',2c'+1\}$, $d=2c'$ and $d=2c'+1$. 

We consider the case $d\not\in\{2c',2c'+1\}$. By condition (\ref{tri_0}), we have
\begin{eqnarray*}
\chi _n(x,y)=0, \quad x\in [2^{-n}d,2^{-n}(d+1)).
\end{eqnarray*}

In the case $d=2c'$, since $x\oplus 2^{-n}=x+ 2^{-n}$, we have 
\begin{eqnarray*}
\chi _n(x,y)=2^n\Longleftrightarrow
\left \{
\begin{array}{l}
x\le y \le x+2^{-n},\\
2^{-n+1}c'\le x,y <2^{-n}(2c'+1).\\
\end{array}
\right.
\end{eqnarray*}
So we have
\begin{eqnarray*}
\chi _n(x,y)=
\left \{
\begin{array}{ll}
2^n & \mathrm{if}\ x\in [2^{-n+1}c',y],\\
0  &  \mathrm{if}\ x\in (y,2^{-n}(2c'+1)).\\
\end{array}
\right.
\end{eqnarray*}

When $d=2c'+1$, by a similar argument to the case $d=2c'$, we have
\begin{eqnarray*}
\chi _n(x,y)=
\left \{
\begin{array}{ll}
2^n & \mathrm{if}\ x\in [2^{-n}(2c'+1),y+2^{-n}],\\
0  & \mathrm{if}\ x\in (y+2^{-n},2^{-n}(2c'+2)).\\
\end{array}
\right.
\end{eqnarray*}
By combining these cases, we have the result.
\qquad \end{proof}

In the last lemma, we show the period of a function $W(k)$.
\begin{lemma}
\label{period_W}
We have that $W(k)$ is a periodic function with period $2^{-a_N}$
for a positive integer $k=\sum_{i=1}^N2^{a_i}.$
\end{lemma}
\begin{proof}
We proceed by induction on $N$. We prove the result for $k=2^{a_1}$.
Let $c\in\mathbb{N}_0$ satisfying $c< 2^{a_1}.$
We have that for $y\in[0,2^{-a_1})$,
\begin{eqnarray*}
& &W(2^{a_1})(y+c 2^{-a_1})
=\int_0^1\chi_{a_1+1}(x,y+c 2^{-a_1})\, dx\\
& &=\int_{c 2^{-a_1}}^{(c+1) 2^{-a_1}}\chi_{a_1+1}(x,y+c 2^{-a_1})\, dx
=\int_0^{2^{-a_1}} \chi_{a_1+1}(z+c 2^{-a_1},y+c 2^{-a_1})\, dz\\
& &=\int_0^{2^{-a_1}} \chi_{a_1+1}(z,y)\, dz
=\int_0^1 \chi_{a_1+1}(z,y)\, dz
=W(2^{a_1})(y).
\end{eqnarray*}
The second and fifth equalities follow from Item 2 of Lemma \ref{period_chi},
the forth equality follows from Item 1 of Lemma \ref{period_chi} and
 the change of variable $x=z+c 2^{-a_1}$ in the third equality.

Now we assume that the lemma holds for the case $k_N=\sum_{i=1}^{N}2^{a_i}$.
We prove the result for the case $k=\sum_{i=1}^{N+1}2^{a_i}$ satisfying $a_{N+1}<a_N$.
By the induction assumption, we have that $W(k_N)(z+d2^{-a_{N}})=W(k_N)(z)$ for $z\in[0,2^{-a_{N}})$ and an integer $d$ satisfying $0\le d<2^{a_N}$.
Let $c\in\mathbb{N}_0$ satisfying $c<2^{a_{N+1}}$.
Then we have that for $y\in[0,2^{-a_{N+1}})$,
\begin{eqnarray*}
& &W(k)(y+c 2^{-a_{N+1}})=W(k_N+2^{a_{N+1}})(y+c 2^{-a_{N+1}})\\
& &=\int_0^1\chi_{a_{N+1}+1}(x,y+c 2^{-a_{N+1}})\cdot W(k_N)(x)\, dx\\
& &=\int_{c 2^{-a_{N+1}}}^{(c+1) 2^{-a_{N+1}}}\chi_{a_{N+1}+1}(x,y+c 2^{-a_{N+1}})\cdot W(k_N)(x)\, dx\\
& &=\int_0^{2^{-a_{N+1}}} \chi_{a_{N+1}+1}(z+c 2^{-a_{N+1}},y+c 2^{-a_{N+1}})\cdot W(k_N)(z+c 2^{-a_{N+1}})\, dz.
\end{eqnarray*}
The third equality follows from Item 2 of Lemma \ref{period_chi}
 and the change of variables $x=z+c 2^{-a_{N+1}}$ in the last equality.
By the induction assumption, we have $W(k_N)(z+c2^{-a_{N+1}})=W(k_N)(z)$ for $z\in [0,2^{-a_{N+1}})$.
Thus we obtain
\begin{eqnarray*}
W(k)(y+c 2^{-a_{N+1}})
=\int_0^{2^{-a_{N+1}}} \chi_{a_{N+1}+1}(z+c 2^{-a_{N+1}},y+c 2^{-a_{N+1}})\cdot W(k_N)(z)\, dz.
\end{eqnarray*}
Then we continue the computation as follows:
\begin{eqnarray*}
& &W(k)(y+c 2^{-a_{N+1}})
=\int_0^{2^{-a_{N+1}}} \chi_{a_{N+1}+1}(z,y)\cdot W(k_N)(z)\, dz\\
&&=\int_0^1 \chi_{a_{N+1}+1}(z,y)\cdot W(k_N)(z)\, dz
=W(k)(y).
\end{eqnarray*}
The first equality follows from Item 1 of Lemma \ref{period_chi} and
the second equality follows from Item 2 of Lemma \ref{period_chi}.
\qquad \end{proof}
\subsubsection{Proof of Lemma \ref{prop_{L^1}_{L^1}}}
We prove Lemma \ref{prop_{L^1}_{L^1}} using the results in the above subsection.
\begin{proof}
Since $\chi_n\ge 0$,
we see that $W(k)\geq 0$ on $[0,1)$ by induction. We omit the details.
We use this property of $W(k)\ge 0$ to prove $\| W(k)\| _{L^p}\le 2^{(1-\frac{1}{p})}$. 

Using H\"{o}lder's inequality, we have 
\begin{eqnarray*}
\| W(k)\| _{L^p}^p=\int_{[0,1)}|W(k)(x)|\cdot | W(k)(x)^{p-1}|\,dx
\leq \| W(k)\| _{L^1} \| W(k)\| _{L^\infty}^{p-1}.
\end{eqnarray*}
Thus we have 
\begin{eqnarray*}
\| W(k)\| _{L^p}\leq \| W(k)\| _{L^1}^{\frac{1}{p}} \| W(k)\| _{L^\infty}^{1-\frac{1}{p}}.
\end{eqnarray*}
Then, if we have
$\| W(k)\| _{L^1}=1$ and $\| W(k)\| _{L^{\infty}} = 2$,
we have 
\begin{eqnarray*}
\| W(k)\| _{L^p}\leq \| W(k)\| _{L^1}^{\frac{1}{p}} \| W(k)\| _{L^\infty}^{1-\frac{1}{p}}= 2^{(1-\frac{1}{p})}.
\end{eqnarray*}
Therefore we prove 
$\| W(k)\| _{L^1}=1$ and $\| W(k)\| _{L^{\infty}} = 2$ for any $k\in\mathbb{N}$ to complete the proof.

We prove the case $k=\sum_{i=1}^N2^{a_i}$ by induction on $N$.
In the case $k=2^{a_1}$, we have
\begin{eqnarray*}
\| W(2^{a_1})\| _{L^1}&=&\int_0^1\int_0^1\chi_{a_1+1}(x,y)\, dydx\\
&=&\int_0^1\int_{\min (x,x\oplus 2^{-a_1-1})}^{\max (x,x\oplus 2^{-a_1-1})}2^{a_1+1}\, dy \, dx=1.
\end{eqnarray*}
The first equality follows from $\chi_{a_1+1}(x,y)\ge 0$.
We prove $\| W(2^{a_1})\| _{L^{\infty}}=2$ to complete the case $k=2^{a_1}$.
By Lemma \ref{period_W}, we have
\begin{eqnarray*}
& &\| W(2^{a_1})\|_{L^{\infty}} 
=\sup_{y\in [0,1)}\left| W(2^{a_1})(y)\right|
=\sup_{y\in [0,2^{-a_1})}\left| W(2^{a_1})(y)\right| .
\end{eqnarray*}
Since $W(2^{a_1})(y)=\int_0^1\chi_{a_1+1}(x,y)\, dx$ and $\chi_{a_1+1}(x,y)\ge 0$, we have
\begin{eqnarray*}
& &\| W(2^{a_1})\|_{L^{\infty}}=\sup_{y\in [0,2^{-a_1})}\int_0^1\chi_{a_1+1}(x,y)\, dx\\
& &=\max\left(\sup_{y\in [0,2^{-a_1-1})}\int_0^1\chi_{a_1+1}(x,y)\, dx,\sup_{y\in [2^{-a_1-1},2^{-a_1})}\int_0^1\chi_{a_1+1}(x,y)\, dx\right) .
\end{eqnarray*}
We calculate the supremum on $[0,2^{-a_1-1})$ and $[2^{-a_1-1},2^{-a_1})$ separately.
We assume that $y\in [0,2^{-a_1-1})$. By Lemma \ref{chi_y}, we have
\begin{eqnarray*}
\int_0^1\chi_{a_1+1}(x,y)\, dx=
\int_{0}^y 2^{a_1+1}\, dx+\int_{2^{-a_1-1}}^{2^{-a_1-1}+y} 2^{a_1+1}\, dx.
\end{eqnarray*}
If we choose $y=2^{-a_1-1}$, we can maximize the right hand side.
Hence we obtain 
\begin{eqnarray*}
\sup_{y\in [0,2^{-a_1-1})}\int_0^1\chi_{a_1+1}(x,y)\, dx
= \int_0^{2^{-a_1}}2^{a_1+1}\, dx_1=2.
\end{eqnarray*}
By the same argument we get the same result in the case $y\in [2^{-a_1-1},2^{-a_1})$. 
We omit the details.
Therefore we have $\| W(2^{a_1})\|_{L^{\infty}}=2$.

Now we assume that for any $k_N=\sum_{i=1}^{N}2^{a_i}$, we have that
$\| W(k_N)\| _{L^1}=1$ and $\| W(k_N)\| _{L^{\infty}}= 2$.
Let $k=k_N+2^{a_{N+1}}$ satisfying $a_N>a_{N+1}$.
We prove that for any $k=k_N+2^{a_{N+1}}$, $\| W(k)\| _{L^1}=1$ and $\| W(k)\| _{L^{\infty}}= 2$.
By the fact that $W(k_N+2^{a_{N+1}})\ge 0$ and Fubini's Theorem, we have 
\begin{eqnarray*}
& &\| W(k_N+2^{a_{N+1}})\| _{L^1}=\int_0^1\int_0 ^1\chi _{a_{N+1}+1}(x,y)W(k_N)(x) \, dxdy=\int_0^1 W(k_N)(x) \, dx.
\end{eqnarray*}
By the fact $W(k_N)\ge 0$ and the assumption on $k_N$, we obtain
$\| W(k_N+2^{a_{N+1}})\| _{L^1}=\| W(k_N)\| _{L^1}=1$.

We prove $\| W(k)\| _{L^{\infty}}= 2$ as follows. 
By the fact $W(k_N+2^{a_{N+1}})\ge 0$ and Lemma \ref{period_W}, we have
\begin{eqnarray*}
\| W(k)\| _{L^{\infty}}&&=\sup_{y\in [0,2^{-a_{N+1}})} W(k_N+2^{a_{N+1}})\\
&&=\max\Big(\sup_{y\in [0,2^{-a_{N+1}-1})} \int_0^1\chi _{a_{N+1}+1}(x,y)W(k_N)(x) \, dx,\\
& &\sup_{y\in [2^{-a_{N+1}-1},2^{-a_{N+1}})} \int_0^1\chi _{a_{N+1}+1}(x,y)W(k_N)(x) \, dx\Big) .
\end{eqnarray*}
We can calculate the supremum by the fact $W(k_N)\ge 0$ and the same method as in the case $k=2^{a_1}$:
\begin{eqnarray*}
\left.
\begin{array}{ll}
\sup_{y\in [0,2^{-a_{N+1}-1})} \int_0^1\chi _{a_{N+1}+1}(x,y)W(k_N)(x) \, dx\\ 
\sup_{y\in [2^{-a_{N+1}-1},2^{-a_{N+1}})} \int_0^1\chi _{a_{N+1}+1}(x,y)W(k_N)(x) \, dx\\
\end{array}
\right\}
=\int_0^{2^{-a_{N+1}}}2^{a_{N+1}+1}W(k_N)(x) \, dx.
\end{eqnarray*}
Then we obtain
\begin{eqnarray*}
\| W(k)\| _{L^{\infty}}
& &= \int_0^{2^{-a_{N+1}}}2^{a_{N+1}+1}W(k_N)(x) \, dx
=2^{a_{N+1}+1}\sum_{i=0}^{2^{a_{N}-a_{N+1}}-1} \int_{i2^{-a_N}}^{(i+1)2^{-a_{N}}}W(k_N)(x) \, dx\\
& &=2^{a_{N+1}+1}\sum_{i=0}^{2^{a_{N}-a_{N+1}}-1} \int_{0}^{2^{-a_{N}}}W(k_N)(x) \, dx=2^{a_{N}+1}\cdot \int_{0}^{2^{-a_{N}}}W(k_N)(x) \, dx.
\end{eqnarray*}
The second equality follows from Lemma \ref{period_W}.
Thus we have
\begin{eqnarray*}
& &\| W(k)\| _{L^{\infty}}
= 2^{a_{N}+1}\cdot \int_0^{2^{-a_{N}}}W(k_N)(x) \, dx\\
& &=2\sum_{i=0}^{2^{a_N}-1}\int_{i2^{-a_N}}^{(i+1)2^{-a_N}} W(k_N)(x) \, dx
=2\cdot\int_0^1 W(k_N)(x) \, dx.
\end{eqnarray*}
The first equality follows from Lemma \ref{period_W}.
By the assumption on $k_N$ and the fact $W(k_N)\ge 0$,
it follows that $\int_0^1 W(k_N)(x) \, dx=\| W(k_N)\| _{L^1}=1$, and hence we obtain
$\| W(k)\| _{L^{\infty}}=2\cdot\int_0^1 W(k_N)(x) \, dx= 2$.
\qquad \end{proof}


\begin{thebibliography}{99}
\bibitem{BJJF} {\sc J. Baldeaux, J. Dick, J. Greslehner, and F. Pillichshammer},
{\em Construction algorithms for higher order polynomial lattice rules},
J. Complexity, 27 (2011), pp.~281--299.
\bibitem{BJJGDF} {\sc J. Baldeaux, J. Dick, G. Leobacher, D. Nuyens, and F. Pillichshammer},
{\em Efficient calculation of the worst-case error and (fast) component-by-component construction of higher order polynomial lattice rules},
Numer. Algorithms, 59 (2012), pp.~403--431.
\bibitem{Dick_periodic} {\sc J. Dick},
{\em Explicit constructions of quasi-Monte Carlo rules for the numerical integration of high-dimensional periodic functions},
SIAM J. Numer. Anal. 45 (2007), pp.~2141--2176.
\bibitem{Dick_nonperiodic} 
{\sc J. Dick},
{\em Walsh spaces containing smooth functions and quasi-Monte Carlo rules of arbitrary high order},
SIAM J. Numer. Anal. 46 (2008), pp.~1519--1553.
\bibitem{Dick_decay2} 
{\sc J. Dick},
{\em On quasi-Monte Carlo rules achieving higher order convergence}, 
Monte Carlo and Quasi-Monte Carlo Methods 2008, Springer, (2009), pp.~73--96.
\bibitem{Dick_decay} 
{\sc J. Dick},
{\em The decay of the Walsh coefficients of smooth functions},
Bull. Austral. Math. Soc., 80 (2009), pp.~430--453.
\bibitem{Dick_garkin} {\sc J. Dick, F. Y. Kuo, Q. T. Le Gia, D. Nuyens, and Ch. Schwab},
{\em Higher order QMC Petrov-Galerkin discretization for affine parametric operator equations with random field inputs},
SIAM J. Numer. Anal. 52 (2014), pp.~2676--2702.
\bibitem{Dick_Pill05} {\sc J. Dick and F. Pillichshammer}, 
{\em Multivariate integration in weighted Hilbert spaces based on Walsh functions and weighted Sobolev spaces},
J. Complexity, 21 (2005), pp.~149--195.
\bibitem{Dick_Pill}
{\sc J. Dick and F. Pillichshammer},
{\em Digital Nets and Sequences, Discrepancy Theory and Quasi-Monte Carlo Integration},
Cambridge University Press, Cambridge, 2010.
\bibitem{Fine} {\sc N. J. Fine},
{\em On the Walsh functions},
Trans. Amer. Math. Soc., 65 (1949), pp.~372--414.
\bibitem{GS14} {\sc R. N. Gantner and Ch. Schwab},
{\em Computational Higher Order Quasi-Monte Carlo Integration},
 Submitted, 2014,
Available at\\
http://www.sam.math.ethz.ch/sam{\_}reports/reports{\_}final/reports2014/2014-25.pdf.
\bibitem{GSY} {\sc T. Goda, K. Suzuki, and T Yoshiki},
{\em The $b$-adic baker's transformation for quasi-Monte Carlo integration using digital nets},
Journal of Approximation Theory, 194 (2015), pp. 62--86,
\bibitem{Harase} {\sc S. Harase},
{\em Quasi-Monte Carlo point sets with small t-values and WAFOM},
Applied Mathematics and Computation, 254 (2015), pp. 318--326.
\bibitem{HO} {\sc S. Harase and R. Ohori},
{\em A search for extensible low-WAFOM point sets},
Arxiv Preprint arXiv:1309.7828v2.
\bibitem{Hickernell} {\sc F. J. Hickernell},
{\em A generalized discrepancy and quadrature error bound},
Math. Comp., 67 (1998), pp.~299--322.
\bibitem{Kuiper} {\sc L. Kuipers and H. Niederreiter},
{\em Uniform Distribution of Sequences}.
John Wiley, New York, 1974. Reprint, Dover Publications, Mineola, NY, 2006.
\bibitem{MSM} {\sc M. Matsumoto, M. Saito, and K. Matoba},
{\em A computable figure of merit for Quasi-Monte Carlo point sets},
Math. Comp., 83 (2014), pp.~1233--1250.
\bibitem{Nied} {\sc H. Niederreiter},
{\em Random Number Generation and Quasi-Monte Carlo Methods},
CBMS-NSF, Philadelphia, Pennsylvania, 1992.
\bibitem{NiedP} {\sc H. Niederreiter and G. Pirsic},
{\em Duality for digital nets and its applications},
Acta Arith., 97 (2001), pp.~173--182.
\bibitem{SWS} {\sc F. Schipp, W. R. Wade, P. Simon, and J. P{\'a}l},
{\em Walsh Series, An Introduction to Dyadic Harmonic Analysis},
Adam Hilger Ltd., Bristol, 1990.
\bibitem{sharygin} {\sc I. F. Sharygin},
{\em A lower estimate for the error of quadrature formulas for certain classes of functions},
Zh. Vychisl. Mat. i Mat. Fiz., 3 (1963), pp.~370--376.
\bibitem{Sloan} {\sc I. H. Sloan and S. Joe},
{\em Lattice Methods for Multiple Integration},
Clarendon Press, Oxford, 1994.
\bibitem{SW98} {\sc I. H. Sloan and H. Wozniakowski},
{\em When are quasi-Monte Carlo algorithms efficient for high-dimensional integrals?},
J. Complexity, 14 (1998), pp.~1--33. 
\end{thebibliography}
\end{document}